\documentclass[12pt]{amsart}

\usepackage[usenames]{color}
\usepackage{fullpage,url,mathrsfs,stmaryrd}

\newcommand{\BA}{{\mathbb{A}}}

\usepackage{amssymb}
\usepackage[all]{xy}
\usepackage{mathrsfs}
\usepackage{enumerate}
\usepackage{amscd}

\usepackage{bbm}

\DeclareMathOperator{\con}{{con}}

\DeclareMathOperator{\Tate}{{Tate}}

\DeclareMathOperator{\perf}{{perf}}

\newcommand{\HHom}{\underline{\on{Hom}}}

\DeclareMathOperator{\cris}{{cris}}

\newcommand{\Good}{{\mathcal Good}}

\newcommand{\cA}{{\mathcal A}}

\newcommand{\cO}{{\mathcal O}}

\newcommand{\fD}{{\mathfrak D}}

\newcommand{\nc}{\newcommand}

\nc\wh{\widehat}

\nc\on{\operatorname}

\nc\Gr{\on{Gr}}

\nc\Fl{\on{Fl}}

\newtheorem{cor}[subsubsection]{Corollary}
\newtheorem{lem}[subsubsection]{Lemma}
\newtheorem{prop}[subsubsection]{Proposition}

\newtheorem{thm}[subsubsection]{Theorem}

\theoremstyle{remark}
\newtheorem{rem}[subsubsection]{Remark}

\newcommand{\BF}{{\mathbb{F}}}
\newcommand{\BN}{{\mathbb{N}}}
\newcommand{\BQ}{{\mathbb{Q}}}
\newcommand{\BR}{{\mathbb{R}}}
\newcommand{\BZ}{{\mathbb{Z}}}

 \DeclareMathOperator{\Spf}{{Spf}}

\newcommand{\limto}{{\displaystyle\lim_{\longrightarrow}}}
\newcommand{\rightlim}{\mathop{\limto}}

\newcommand{\leftlim}{\mathop{\displaystyle\lim_{\longleftarrow}}}
\newcommand{\limfromn}{\leftlim\limits_{\raise3pt\hbox{$n$}}}
\newcommand{\limton}{\rightlim\limits_{\raise3pt\hbox{$n$}}}

\newcommand{\rightlimit}[1]{\mathop{\lim\limits_{\longrightarrow}}\limits%
                    _{\raise3pt\hbox{$\scriptstyle #1$}}}

\newcommand{\leftlimit}[1]{\mathop{\lim\limits_{\longleftarrow}}\limits%
                    _{\raise3pt\hbox{$\scriptstyle #1$}}}

\newcommand{\epi}{\twoheadrightarrow}
\newcommand{\iso}{\buildrel{\sim}\over{\longrightarrow}}
\newcommand{\mono}{\hookrightarrow}

\DeclareMathOperator{\Coker}{{Coker}}

\DeclareMathOperator{\End}{{End}}

\DeclareMathOperator{\gr}{{gr}} 
\DeclareMathOperator{\Hom}{{Hom}}

\DeclareMathOperator{\Ker}{{Ker}} 
\DeclareMathOperator{\im}{{Im}}

 \DeclareMathOperator{\op}{{op}}

\DeclareMathOperator{\Spec}{{Spec}}

\DeclareMathOperator{\Gal}{{Gal}}

\theoremstyle{definition}

\newtheorem{defin}[subsubsection]{Definition}

\numberwithin{equation}{section}

\newcommand{\Fr}{\operatorname{Fr}}

\begin{document}
\title[On a theorem of Scholze-Weinstein]{On a theorem of Scholze-Weinstein}
\author{Vladimir Drinfeld}
\address{University of Chicago, Department of Mathematics, Chicago, IL 60637}
\dedicatory{Dedicated to the memory of Galim Mustafin}

\begin{abstract}
Let $G$ be the Tate module of a $p$-divisble group $H$ over a perfect field $k$ of characteristic $p$. A theorem of Scholze-Weinstein describes $G$ (and therefore $H$ itself) in terms of  the Dieudonn\'e module of~$H$; more precisely, it describes $G(C)$ for ``good" semiperfect $k$-algebras $C$ (which is enough to reconstruct $G$). 

In these notes we give a self-contained proof of this theorem and explain the relation with the classical descriptions of the Dieudonn\'e functor from Dieudonn\'e modules to $p$-divisible groups.
\end{abstract}


\maketitle

\section{Introduction}   \label{s:intro}
Fix a prime $p$. Recall that an $\BF_p$-algebra is said to be \emph{perfect} (resp.~\emph{semiperfect}) if the Frobenius homomorphism $\Fr:C\to C$ is bijective (resp.~surjective).

The goal of these notes is to give a self-contained proof of Theorem~\ref{t:SW}. This theorem is due to P.~Scholze and J.~Weinstein (it is a very special case of the theory developed in \S 4 of~\cite{SW}). If $H$ is a $p$-divisble group over a perfect field $k$ of characteristic $p$ and $G$ is its Tate module, the theorem describes $G$ (and therefore $H$ itself) in terms of  the Dieudonn\'e module of~$H$; more precisely, it describes $G(C)$ for ``good" semiperfect $k$-algebras $C$ (which is enough to reconstruct $G$). The description is in terms of Fontaine's ring $A_{\cris}(C)$.

\S\ref{ss:N pronilpotent} and \S\ref{ss:conceptually better} are influenced by Fontaine's book \cite{F77}; they explain the relation between Theorem~\ref{t:SW} and the classical descriptions of the Dieudonn\'e equivalence \cite{Dem,F77}. The idea of \S\ref{ss:conceptually better} is to switch from $A_{\cris}(C)$ to a more manageable $W(k)$-module $\overline{M} (C)$, which was introduced by Fontaine \cite{F77} (under a different name); the definition of $\overline{M} (C)$ is given in \S\ref{sss:bar M(C)}. Probably this idea is somewhat similar to \cite[\S 4.2]{FF}. Our \S\ref{ss:N pronilpotent} is a ``baby version" of \S\ref{ss:conceptually better}; instead of $\overline{M} (C)$ we work there with a certain submodule $M(C)\subset\overline{M} (C)$.

 \S\ref{ss:thekeyexample} and formula \eqref{e:explicit c} reflect some exercises, which I had to do in order to under\-stand~\cite{SW}.

\medskip

I thank P.~Scholze and J.~Weinstein for valuable advice and references.
The author's research was partially supported by NSF grant DMS-1303100.

\section{Recollections on Fontaine's functor $A_{\cris}$}   \label{s:Fontaine}
In \S\ref{ss:Fontaine's def}-\ref{ss:F/p} we follow \cite[\S 2.2]{F94} and \cite[\S 4]{SW}, but the proof of the important Proposition~\ref{p:SW on F/p} is different from the one given in \cite{SW}. The material of \S\ref{ss:M(C)} is influenced by Fontaine's book \cite{F77}; it is used in the proof of Proposition~\ref{p:N pronilpotent}. 

We will use the following notation: we write $W(R)$ for the ring of $p$-typical Witt vectors of a ring $R$, and 
for $a\in R$ the Teichm\"uller element $(a,0,0,\ldots )\in W(R)$ is denoted by $[a]$.

\subsection{The definition of $A_{\cris}$}   \label{ss:Fontaine's def}

By the \emph{Fontainization} of an $\BF_p$-algebra $C$ we mean the perfect topological $\BF_p$-algebra
\[
C^{\flat}:=\underset{\longleftarrow}\lim (\ldots\overset{\Fr}\longrightarrow C\overset{\Fr}\longrightarrow C\overset{\Fr}\longrightarrow C).
\]
Thus an element of $C^{\flat}$ is a sequence $(c_0,c_1,\ldots )$ of elements of $C$ such that $c_{n+1}^p=c_n$ for all~$n$. Define $\nu_n :C^{\flat}\to C$ by $\nu (c_0,c_1,\ldots ):=c_n\,$; then $\nu_n=\nu\circ\Fr_{C^\flat}^{-n}$.

Now suppose that $C$ is semiperfect. Then the homomorphisms $\nu_n :C^{\flat}\to C$ are surjective. Fontaine defined $A_{\cris}(C)$ to be the $p$-adic completion of the PD hull\footnote{The definitions of PD thickening and PD hull include the condition $\gamma_m(p)=p^m/m!$.} of the epimorphism $W(C^{\flat})\epi C$ induced by $\nu_0:C^{\flat}\epi C$. 
Despite the fact that the definitions of $C^{\flat}$ and $W$ involve projective limits, one has the following

\begin{prop}   \label{p:finitistic}
For any $n\in\BN$, the functor $C\mapsto A_{\cris} (C)/p^nA_{\cris} (C)$ commutes with filtered inductive limits. 
\end{prop}

\begin{proof}
The canonical homomorphism $f:W(C^\flat )\to A_{\cris} (C)/p^nA_{\cris} (C)$ factors through $W_n(C^\flat )$. Moreover, if $u\in\Ker (C^\flat\overset{\nu_0}\epi C)$ then
\[
f(V^i[u^{p^n}])=f(p^i[u^{p^{n-i}}])=p^i\cdot (p^{n-i})!\cdot\gamma_{p^{n-i}}([u])=0\quad \mbox{ for }0\le i<n. 
\]
Therefore $f$ factors through $W_n(C^\flat/\Ker\nu_n)$. So $A_{\cris} (C)/p^nA_{\cris} (C)$ is the PD hull of the epimorphism
\begin{equation}  \label{e:economic epi}
W_n(C^\flat/\Ker\nu_n)\epi C
\end{equation}
induced by $\nu_0: C^\flat\epi C$. The functor $C\mapsto W_n(C^\flat/\Ker\nu_n)$ commutes with filtered inductive limits.
\end{proof}

\begin{rem}
The isomorphism $C^\flat/\Ker\nu_n\iso C$ induced by $\nu_n: C^\flat\epi C$ transforms the map \eqref{e:economic epi} into the composed map
$W_n(C)\epi C\overset{\Fr^n}\epi C$.
\end{rem}

\subsubsection{The canonical epimorphism $A_{\cris} (C)\epi W(C)$}  \label{sss:beta}
The ideal $\Ker (W(C)\epi C)$ has a canonical PD structure, namely $\gamma_n (Va):=\frac{p^{n-1}}{n!}\cdot V(a^n)$ for $a\in W(C)$. So the canonical epimorphism $W(C^\flat )\epi W(C)$ uniquely extends to a PD homomorphism 
\[
\beta :A_{\cris} (C)\epi W(C).
\]

\subsection{The universal property}
Let $C$ be a semiperfect $\BF_p$-algebra and $n\in\BN$. Let $\cA_n$ be the category of PD thickenings $\tilde C\epi C$ such that $p^n=0$ in $\tilde C$. 

\begin{prop}   \label{p:Fontaine universality}
$A_{\cris} (C)/p^nA_{\cris} (C)$ is an initial object of $\cA_n\,$.
\end{prop}

The proof uses the following

\begin{lem}   \label{l:Fontaine universality}
Let $\pi :\tilde C\epi C$ be an object of $\cA_n$ and $I:=\Ker\pi$. Then

(i) for every $x\in I$ one has $x^p\in pI$;

(ii) if $a\in\tilde C$ and $x\in I$ then $(a+x)^{p^j}-a^{p^j}\in p^jI$ for all $j$; in particular, $(a+x)^{p^n}=a^{p^n}$;

(iii) the ring homomorphism $w_n: W_n(\tilde C)\to\tilde C$ defined by $w_n(a_0,\ldots , a_n):=\sum\limits_{j=0}^n p^ja_j^{p^{n-i}}$ factors as
\begin{equation}  \label{e: bar w_n}
W_n(\tilde C)\epi W_n(C)\overset{\bar w_n}\longrightarrow\tilde C.
\end{equation}
\end{lem}

\begin{proof}
Since $I$ has a PD structure we have (i). Statements (ii) and (iii) follow from (i).
\end{proof}

\begin{proof}[Proof of Proposition~\ref{p:Fontaine universality}]
It is clear that $A_{\cris} (C)/p^nA_{\cris} (C)$ is an object of $\cA_n\,$. To prove that this object is initial, one has to show that if
$\pi :\tilde C\epi C$ is an object of $\cA_n$ then there is a unique homomorphism $f:W(C^\flat )\to\tilde C$ such that the following diagram commutes:
\[
\CD
W(C^\flat )  @>f>> \tilde C   \\
@VVV    @VV\pi V   \\
C^\flat   @>\nu>> C
\endCD
\]

Define $f$ to be the composition
\[
W(C^\flat )\overset{\alpha_n}\longrightarrow W(C)\overset{\bar w_n}\longrightarrow\tilde C,
\]
where $\alpha_n$ is induced by $\nu_0\circ\Fr^{-n}: C^\flat\to C$ and $\bar w_n$ is as in \eqref{e: bar w_n}. Then $f$ has the required property.

Now let $f:W(C^\flat )\to\tilde C$ be any homomorphism with the required property.
If $u\in C^\flat$ and $[u]\in W (C^\flat )$ is the corresponding Teichm\"uller element then
Lemma~\ref{l:Fontaine universality}(ii) implies that  
$$f([u]))=f([u^{p^{-n}}])^{p^n}=a^{p^n},$$ where $a$ is any element of $\tilde C$ such that $\pi (a)=\nu_0 (u^{p^{-n}})$. 
Since $$\Ker (W (\tilde C))\to W_n(\tilde C))=p^n W (\tilde C)\subset\Ker f$$ and $W_n(\tilde C)$ is generated by the Teichm\"uller elements, we see that $f$ is unique.
\end{proof}

\subsection{Commutation of $A_{\cris}$ with a certain type of base change}

\begin{lem}  \label{l:W(quotient)}
Let $B$ be a semiperfect $\BF_p$-algebra. Let $b_j,b_j'\in B$, where $j$ runs through some set $J$. Let $I\subset B$ be
the ideal generated by the elements $b_j-b'_j$, $j\in J$. Then for every $n\in\BN$ the ideal 
$\Ker (W_n(B)\epi W_n(B/I))$ is generated by $V^m ([b_j]-[b'_j])$, where $j\in J$, $0\le m<n$.
\end{lem}

\begin{proof}
By induction, it suffices to check that the ideal $\Ker (W_n(B)\to W_{n-1}(B)\times W_n(B/I))$ is generated by $V^{n-1}([b_j]-[b'_j])$, where $j\in J$. This follows from semiperfectness and the identity $x\cdot V^{n-1} y=V^{n-1}( F^{n-1} x\cdot y )$ in $W(B)$.
\end{proof}

\begin{lem}   \label{l:W(tensor product)}
Let $C$ be an $\BF_p$-algebra. Let $C_1$ and $C_2$ be semiperfect $C$-algebras. Then for every $n\in\BN$ the map $W_n(C_1)\otimes_{W_n(C)}W_n(C_2)\to W_n(C_1\otimes_CC_2)$ is an isomorphism. 
\end{lem}

\begin{proof}
Write $C_j=B_j/I_j$, where $B_j$ is perfect (e.g., one can take $B_j=C_j^\flat$). It is easy to see that $W_n(B_1\otimes B_2)= W_n(B_1)\otimes W_n(B_2)$.
The ideal $\Ker (B_1\otimes B_2\epi C_1\otimes_CC_2)$ is generated by elements of the form $b_1\otimes 1-1\otimes b_2$, where $b_j\in B_j$ are such that the image of $(b_1,b_2)$ in $C_1\times C_2$ is contained in $\im (C\to C_1\times C_2)$. So by Lemma~\ref{l:W(quotient)}, $W_n(C_1\otimes_CC_2)$ is the quotient of 
$W_n(B_1)\otimes W_n(B_2)$ by the ideal generated by elements of the form $V^m[b_1]\otimes 1-1\otimes V^m[b_2]$, where $b_1,b_2$ are as above and $0\le m<n$. These elements have zero image in $W_n(C_1)\otimes_{W_n(C)}W_n(C_2)$.
\end{proof}

\begin{lem}   \label{l:PD hulls}
The formation of PD hulls commutes with flat base change.
\end{lem}

\begin{proof}
This  was proved by Berthelot \cite[Prop.~I.2.7.1]{B74}.
\end{proof}

\begin{prop}   \label{p:flat perfect base change}
Let $B$ be an $\BF_p$-algebra and $B'$ a flat $B$-algebra. Suppose that $B$ and $B'$ are perfect. Let $C$ be a semiperfect $B$-algebra. Then 
the canonical map $$A_{\cris}(C)\hat\otimes_{W(B)}W(B')\to A_{\cris}(C\otimes_BB')$$ is an isomorphism.
\end{prop}

\begin{proof}
It suffices to show that for each $n\in\BN$ the map
\begin{equation}   \label{e:the map mod p^n}
(A_{\cris}(C)/p^nA_{\cris}(C))\otimes_{W_n(B)}W_n(B')\to A_{\cris}(C\otimes_BB')/p^nA_{\cris}(C\otimes_BB')
\end{equation}
is an isomorphism. In the proof of Proposition~\ref{p:finitistic} it was shown that
$A_{\cris} (C)/p^nA_{\cris} (C)$ is the PD hull of the epimorphism
\begin{equation}  \label{e:2economic epi}
W_n(C^\flat/\Ker\nu_n)\epi C
\end{equation}
induced by $\nu_0: C^\flat\epi C$. One also has a similar description of $A_{\cris}(C\otimes_BB')/p^nA_{\cris}(C\otimes_BB')$; combining it with Lemma~\ref{l:W(tensor product)}, we see that $A_{\cris}(C\otimes_BB')/p^nA_{\cris}(C\otimes_BB')$  is the PD hull of the epimorphism
\[
W_n(C^\flat/\Ker\nu_n)\otimes_{W_n(B)}W_n(B')\epi C\otimes_{W_n(B)}W_n(B')=C\otimes_BB'
\]
obtained from \eqref{e:2economic epi} by base change via $f:W_n(B)\to W_n(B')$. It is easy to see that $f$ is flat, so the map \eqref{e:the map mod p^n} is an isomorphism by Lemma~\ref{l:PD hulls}.
\end{proof}

\subsection{The functor $A_{\cris}$ for schemes}   \label{ss:Acris for schemes}
\begin{prop}   \label{p:etale local}
For any $n\in\BN$, the functor $C\mapsto A_{\cris} (C)/p^nA_{\cris} (C)$ commutes with etale localization.
\end{prop}

\begin{proof}   
Follows from Proposition~\ref{p:Fontaine universality} combined with Lemma~\ref{l:PD hulls}.
\end{proof}

Let $X$ be a semiperfect $\BF_p$-scheme. Then one defines a $p$-adic formal scheme $A_{\cris}(X)$ as follows: its underlying topological space is that of $X$, and its structure sheaf is  the $p$-adic completion of the PD hull of the surjection $W(\cO_X)\epi\cO_X$. By Proposition~\ref{p:etale local}, $A_{\cris}(X)$ is a scheme, and the functor $X\mapsto A_{\cris}(X)$ commutes with etale localization.

\subsection{The key example} \label{ss:thekeyexample}
\subsubsection{The ring $C$}    \label{sss:1thekeyexample}
Let $B$ be a perfect $\BF_p$-algebra. Let $B[x_1^{p^{-\infty}},\ldots ,x_n^{p^{-\infty}}]$ denote the perfection of $B[x_1,\ldots ,x_n]$; in other words,
an element of $B[x_1^{p^{-\infty}},\ldots ,x_n^{p^{-\infty}}]$ is a finite sum
\begin{equation}   \label{e:perfectpolyn}
\sum_{\alpha\in\BZ_+[1/p]^n} b_\alpha x^\alpha , \quad b_\alpha\in B,
\end{equation}
where $\BZ_+[1/p]:=\{\alpha\in\BZ [1/p]\, |\,\alpha\ge 0\}$ and for $\alpha=(\alpha_1,\ldots,\alpha_n)\in\BZ_+[1/p]^n$ we set
\[
x^\alpha:=\prod_{i=1}^n x_i^{\alpha_i}.
\]
Now let  $C:=B[x_1^{p^{-\infty}},\ldots ,x_n^{p^{-\infty}}]/(x_1,\ldots , x_n)$. Then $C$ is semiperfect. In \S\ref{sss:2thekeyexample}-\ref{sss:3thekeyexample} we will describe $C^\flat$, $W(C^\flat )$, $W(C)$, and $C_{\cris}\,$ for such $C$. This class of semiperfect $\BF_p$-algebras is important because of Proposition~\ref{p:Tate as schemes} and the following remark.

\begin{rem} \label{r:why important}
Let $X$ be an affine scheme etale over $\BA^n_k$, where $k$ is a perfect field of characteristic $p$. Let $X_{\perf}$ be the perfection of $X$.
Then for each $m\in\BN$ the fiber product (over $X$) of $m$ copies of $X_{\perf}$ is isomorphic to $\Spec C$, where $C$ is as in \S\ref{sss:1thekeyexample}. Let us note that fiber products of this type appear in \cite{BMS}. 
\end{rem}

\subsubsection{$C^\flat$, $W(C^\flat )$, and $W(C)$}   \label{sss:2thekeyexample}
Clearly $C^\flat$ is the ring of formal series \eqref{e:perfectpolyn} such that the set $$\{\alpha\in\BZ_+[1/p]^n \,|\,b_\alpha\ne 0\}$$ is discrete in $\BR^n$. 
The Witt ring $W(C^\flat )$ identifies with the ring of formal series 
\begin{equation}   \label{e:W(Cflat )}
\sum_{\alpha\in\BZ_+[1/p]^n} a_\alpha x^\alpha , \quad a_\alpha\in W(B)
\end{equation}
such that $a_\alpha\to 0$ when $\alpha$ runs through any bounded subset of $\BZ_+[1/p]^n$ (boundedness in the sense of $\BR^n$). 
Let is note that in formula \eqref{e:W(Cflat )} and similar formulas below $x^\alpha$ 
really means $X_1^{\alpha_1}\cdot\ldots\cdot X_n^{\alpha_n}$, where $X_i=[x_i]$ is the Teichm\"uller representative.

By Lemma~\ref{l:W(quotient)}, $W(C)$ is the quotient of $W(C^\flat )$ by the ideal topologically generated by $p^mx_i^{p^{-m}}$, $m\ge 0$, $i=1,\ldots , n$.

\subsubsection{The ring $A_{\cris} (C)$}    \label{sss:3thekeyexample}
For a real number $y\ge0$ let $(y!)_p$ be the maximal power of $p$ dividing $\lfloor y\rfloor!$, where $\lfloor y\rfloor$ is the integral part of $y$; equivalently,
$$(y!)_p=p^{s (y)}, \mbox{ where }s(y):=\sum_{j=1}^\infty \lfloor y/p^j\rfloor .$$ 
For $\alpha\in\BZ_+[1/p]^n$ let 
$$(\alpha !)_p:=\prod_{i=1}^n (\alpha_i !)_p\,.$$
Then $A_{\cris} (C)$ identifies with the ring of formal series
\begin{equation}   \label{e:explicit Acris}
\sum_{\alpha\in\BZ_+[1/p]^n} a_\alpha \frac{x^\alpha}{(\alpha !)_p} , \quad a_\alpha\in W(B), \quad a_\alpha\to 0.
\end{equation}

\subsection{The homomorphism $F:A_{\cris}(C)\to A_{\cris}(C)$}   \label{ss:F/p}
Let $C$ be a semiperfect $\BF_p$-algebra. By functoriality, the endomorphism $\Fr_C\in\End C$ induces endomorphisms $F\in\End W(C^\flat )$ and $F\in\End A_{\cris}(C)$. There is no map $V:A_{\cris}(C)\to A_{\cris}(C)$; instead, the following proposition gives a partially defined map $V^{-1}=p^{-1}F$.

\begin{prop}   \label{p:SW on F/p}
Let $I_{\cris}(C)\subset A_{\cris}(C)$ be the kernel of the canonical epimorphism $A_{\cris}(C)\epi C$. Then there exists a unique $F$-linear map $F':I_{\cris}(C)\to A_{\cris}(C)$ such that
\[
F(b)=F'(pb) \quad\mbox{ for } b\in A_{\cris}(C),
\]
\[
F'(\gamma_n(b))=\frac{p^{n-1}}{n!}\cdot F'(b)^n \quad\mbox{ for } b\in I_{\cris}(C),\, n\in\BN,
\]
\[
F'(b)=(p-1)!\cdot \gamma_p(b)+\delta (b) \quad\mbox{ for } b\in\Ker (W(C^\flat )\epi C^\flat\overset{\nu_0}\epi C),
\]
where $\delta: W(C^\flat )\to W(C^\flat )$ is defined by 
\begin{equation}   \label{e:Joyal's delta}
\delta (b):=\frac{F(b)-b^p}{p}. 
\end{equation}
\end{prop}

This is essentially \cite[Lemma~4.1.8]{SW}, but the proof given in \S\ref{sss:proof of SW on F/p} below is different from the one in \cite{SW}.

\subsubsection{A general setup}
Let $\BZ_{(p)}:=\BZ_p\cap\BQ$. 
Let $B$ be a $\BZ_{(p)}$-algebra and $I\subset B$ an ideal such that $p\in I$. Let $R$ be a ring equipped with an additive map $\psi :I\to R$ such that
\begin{equation}   \label{e:1psi}
\psi (b_1b_2)=\psi (pb_1)\psi (b_2) \quad\mbox{ for } b_1\in B, \, b_2\in I,
\end{equation}
\begin{equation}   \label{e:2psi}
\psi (p)=1.
\end{equation}
Then the map $\varphi :B\to R$ defined by
\[
\varphi (b)=\psi (pb)
\]
is a unital ring homomorphism. Moreover, $\psi$ is a $B$-module homomorphism if the $B$-module structure on $R$ is defined using $\varphi$.

Let $(B',I')$ be the PD hull\footnote{The definitions of PD thickening and PD hull include the condition $\gamma_m(p)=p^m/m!$.} of $(B,I)$. If $R$ is $\BZ_{(p)}$-flat then $pR$ is a PD ideal in $R$, so $\varphi$ extends uniquely to a PD morphism
$\varphi' :(B',I')\to (R,pR)$. Here is a statement in this spirit \emph{without} assuming $R$ to be $\BZ_{(p)}$-flat.

\begin{lem}   \label{l:PD lemma}
(i) Equip $R$ with multiplication $x*y:=pxy$ and divided power operations $\gamma_n(x):=\frac{p^{n-1}}{n!}\cdot x^n$. Then
$R$ is a (non-unital) PD algebra\footnote{This means that $R$ is a PD ideal of the $B$-algebra $B\oplus R$ obtained from $R$ by formally adding the unit.} over $B$.

(ii) The map $\psi :I\to (R,*)$ is a homomorphism of $B$-algebras.

(iii) There exists a unique homomorphism $\psi':I'\to (R,*)$ of PD algebras over $B$ extending $\psi :I\to R$.

(iv) One has $\psi'(b_1b_2)=\psi'(pb_1)\psi'(b_2)$ for $b_1\in B'$, $b_2\in I'$.

(v) The map $\varphi':B'\to R$ defined by $\varphi'(b):=\psi'(pb)$ is a ring homomorphism. Moreover, $\psi':I'\to B'$ is a $B'$-module homomorphism if the $B'$-module structure on $R$ is defined using~$\varphi'$.
\end{lem}

\begin{proof}
Checking (i)-(ii) is straightforward. Statement (iii) follows from (ii) and the construction of $(B',I')$ in the proof of Theorem I.2.3.1 of \cite{B74}. 

Since the map $B/I\to B'/I'$ is an isomorphism, it suffices to check (iv) if $b_1\in I'$ and if $b_1$ belongs to the image of $B$. In these cases (iv) follows from (iii).

Statement (v) follows from (iv).
\end{proof}

\subsubsection{Proof of Proposition~\ref{p:SW on F/p}}   \label{sss:proof of SW on F/p}
Uniqueness is clear. To prove existence, we apply Lem\-ma~\ref{l:PD lemma} as follows.

Let $B:=W(C^\flat )$. Let $I\subset B$ be the kernel of the composed map $B\epi C^\flat\overset{\nu_0}\epi C$. Let $(B',I')$ be the PD hull of $(B,I)$. Let $R:=B'$. 
For $b\in I$ set $\psi (b):=(p-1)!\cdot \gamma_p(b)+\delta (b)$, where
$\delta (b)\in B$ is given by \eqref{e:Joyal's delta}.

One checks that $\psi (pb)=F(b)$ for $b\in B=W(C^\flat )$. One also checks that $\psi$ is additive and satisfies \eqref{e:1psi}-\eqref{e:2psi}. Applying Lemma~\ref{l:PD lemma}, we get $\psi':I'\to B'$. Passing to $p$-adic completions, one gets the desired map $F':I_{\cris}(C)\to A_{\cris}(C)$. \qed

\medskip

Let us note that for $x\in W(C^\flat )$ one has
\begin{equation}  \label{e:F'V}
F'Vx=F'(pF^{-1}x)=F(F^{-1}x)=x.
\end{equation}

\begin{prop}   \label{p:beta F'}
Let $\beta :A_{\cris} (C)\epi W(C)$ be the canonical epimorphism constructed in \S\ref{sss:beta}.
Then the following diagram commutes:
\[
\CD
I_{\cris} (C)\  @>F' >>  A_{\cris} (C)   \\
@V\beta VV    @VV\beta  V   \\
V (W(C))    @> {V^{-1}}  >> W(C)
\endCD
\]
\end{prop}

\begin{proof}
By \eqref{e:F'V}, the identity $\beta (F'(b))=V^{-1}(\beta (b))$ holds if $b\in\Ker (W(C^\flat )\epi C^\flat )$; it also holds if $b$ is the Teichm\"uller representative of an element of $\Ker (C^\flat\overset{\nu_0}\epi C)$. So it holds for any 
element $b\in\Ker (W(C^\flat )\epi C^\flat\overset{\nu_0}\epi C)$. This implies that
$$\beta (F'(\gamma_n(b)))=V^{-1}(\beta (\gamma_n(b)))\mbox{ for all } n\in\BN, b\in\Ker (W(C^\flat )\epi C^\flat\overset{\nu_0}\epi C).$$ 
This is enough because the ideal $I_{\cris} (C)\subset A_{\cris} (C)$ is topologically generated by $\gamma_n (b)$, where $b$ and $n$ are as above.
\end{proof}

\begin{cor}  \label{c:(F')^n}
Let $\beta_n:A_{\cris} (C)\to W_n(C)$ be the map induced by $\beta :A_{\cris} (C)\to W (C)$.
Then for every $n\in\BN$ the operator $F':I_{\cris} (C)\to A_{\cris} (C)$ maps $\Ker\beta_{n}$ to $\Ker\beta_{n-1}$. 
So we have the map $(F')^n:\Ker\beta_n\to A_{\cris} (C)$.  \qed
\end{cor}

\begin{rem}  \label{r:(F')^n on Witt}
Using \eqref{e:F'V}, we see that the restriction of $(F')^n$ to $\Ker (W(C^\flat)\epi W_n(C))$ is very simple: namely, $(F')^nV^nu=u$ for $u\in W(C^\flat )$ and $(F')^nV^i[c]=(p^{n-i}-1)!\cdot\gamma_{p^{n-i}}([c])$ for $c\in \Ker (C^\flat\overset{\nu_0}\epi C)$.
\end{rem}

\subsection{The module $M(C)$ and the map $M(C)\to A_{\cris} (C)$}  \label{ss:M(C)}
In this section we define and study a $W(C^\flat )[F,V]$-module $M(C)$ and a canonical $W(C^\flat )[F]$-morphism
$M(C)\to A_{\cris} (C)$. This material is used in the proof of Proposition~\ref{p:N pronilpotent}. 

Let us note that $M(C)$ is a submodule of the module $\overline{M} (C)$ introduced in \S\ref{sss:bar M(C)} below; the latter goes back to Fontaine's book \cite{F77}.

\subsubsection{Definition of $M(C)$} \label{sss:M(C)}

For any $\BF_p$-algebra $B$ set
\[
W(B)[V^{-1}]:=\underset{\longrightarrow}\lim (W(B)\overset{V}\longrightarrow W(B)\overset{V}\longrightarrow\ldots )=\bigcup_n V^{-n}W(B).
\]
We equip $W(B)$ and $W(B^\flat )$ with their natural topologies (they come from the presentation of $W(B)$ as a projective limit of $W_m(B)$ and the presentation of $W(B^\flat )$ as a projective limit of $W(B^\flat/\Ker\nu_n)$. We equip $W(B)[V^{-1}]$ with the inductive limit topology. Then
$W(B)[V^{-1}]$ is a complete topological $W(B^\flat )$-module\footnote{An element $u\in W(B^\flat )$ acts on $W(B)[V^{-1}]$ by 
$V^{-n}x\mapsto V^{-n}((F^{-n}u)\cdot x)$.} equipped with operators $F,V:W(B)[V^{-1}]\to W(B)[V^{-1}]$ satisfying the usual identities.  Each element of $W(B)[V^{-1}]$ has a unique expansion 
\[
\sum_{m=-\infty}^{\infty}V^m[x_m],
\]
where $x_m\in B$ and $x_m=0$ for sufficiently negative $m$.

Recall that $C^\flat$ is the projective limit of $C^\flat/\Ker\nu_n $, where $\nu_n$ is as in \S\ref{ss:Fontaine's def}.
We equip $C^\flat$ with the projective limit topology.

The projective limit of the topological modules $W(C^\flat/\Ker\nu_n )[V^{-1}]$, $n\in\BN$,  is a topological $W(C^\flat)$-module equipped with operators $F,V$ satisfying the usual identities. Let $M(C)$ be the preimage of $W(C)\subset W(C)[V^{-1}]$ in this projective limit. Again, $M(C)$ is a topological $W(C^\flat)$-module equipped with operators $F,V:M(C)\to M(C)$. The map $\nu_0: C^\flat\epi C$ induces a canonical epimorphism $M(C)\epi W(C)$. 

\begin{prop}   \label{p:M(C)}
(i) $M(C)$ is complete with respect to the above topology.

(ii) Each element of $M(C)$ has a unique expansion
\begin{equation}  \label{e:M(C) expansion}
\sum_{n=-\infty}^{\infty}V^n[x_n], \mbox{ where } x_n\in C^\flat, \, x_n\in\Ker (C^\flat\overset{\nu_0}\epi C)\mbox{ for }n<0, \,\lim\limits_{n\to -\infty} x_n= 0.
\end{equation}

(iii) $\Ker (M(C)\overset{p}\longrightarrow M(C))=0$.

(iv) An element  \eqref{e:M(C) expansion} belongs to $p^rM(C)$ if and only if $x_n\in\Ker (C^\flat\overset{\nu_r}\epi C)$ for $n<0$ and $x_n\in\Ker (C^\flat\overset{\nu_{r-n}}\epi C)$ for $0\le n<r$.

(v) The topology on $M(C)$ defined above is equal to the $p$-adic topology of $M(C)$.

(vi) The $W(C^\flat )$-module $M(C)/pM(C)$ is canonically isomorphic to the associated graded of the filtration
$$C^\flat\supset\Ker\nu_1\supset\Ker\nu_2\supset\ldots\;.$$ The isomorphism is as follows: the map 
$C^\flat/\Ker\nu_1\to M(C)/pM(C)$ is induced by the map
\[
C^\flat\to M(C), \quad c\mapsto [c]\in W(C^\flat )\subset M(C),
\]
and for $r>0$ the map $\Ker\nu_r/\Ker\nu_{r+1}\to M(C)/pM(C)$ is induced by the map
\[
\Ker\nu_r\to M(C), \quad c\mapsto p^{-r}[c]=V^{-r}[c^{p^{-r}}]\in M(C).
\]
\end{prop}

The proof of the proposition is straightforward and left to the reader.

\subsubsection{The submodule $M_0(C)\subset M(C)$}    \label{sss:M_0(C)}
Let $M_0(C)$ be the set of all $x\in M(C)$ such that $V^nx\in W(C^\flat )$ (or equivalently, $p^nx\in W(C^\flat ))$ for some $n\ge 0$. Clearly $M_0(C)$ is a $W(C^\flat )[F,V]$-submodule of $M(C)$. One can also think of $M_0(C)$ as a subset of $W(C^\flat )[V^{-1}]=W(C^\flat )[p^{-1}]$; an element $x\in W(C^\flat )[p^{-1}]$ belongs to $M_0(C)$ if and only if for some (or all) $n\ge 0$ such that $V^nx\in W(C^\flat )$ one has $V^nx\in\Ker\beta_n$, where $\beta_n: W(C^\flat )\epi W_n(C)$ is as in Corollary~\ref{c:(F')^n}.

\begin{prop}   \label{p:M_0(C)}
$M(C)$ is the $p$-adic completion of $M_0(C)$. 
\end{prop}

\begin{proof}
One has $p^nM(C)\cap M_0(C)=p^nM_0(C)$. It remains to use Proposition~\ref{p:M(C)}(i,v) and density of $M_0(C)$ in $M(C)$.
\end{proof}

\subsubsection{The map $f:M(C)\to A_{\cris}(C)$}    \label{sss:M(C) to Acris(C)}
Define $f_0:M_0(C)\to A_{\cris}(C)$ as follows: $f_0(x):=(F')^nV^nx$, where $n\ge 0$ is so big that $V^nx\in W(C^\flat )$. The map $f_0$ is well-defined: indeed, $(F')^nV^nx$ is defined by Corollary~\ref{c:(F')^n} (because $V^nx\in\Ker\beta_n$) and moreover, $(F')^nV^nx$ does not depend on $n$ by \eqref{e:F'V}. The map $f_0$ is a $W(C^\flat )[F]$-module homomorphism. By Proposition~\ref{p:M(C)}(v), it uniquely extends to a $W(C^\flat )[F]$-module homomorphism  $f:M(C)\to A_{\cris}(C)$.

Recall that each element of $M(C)$ has a unique expansion
\begin{equation}  \label{e:2M(C) expansion}
\sum_{n=-\infty}^{\infty}V^n[x_n], \mbox{ where } x_n\in C^\flat, \, x_n\in\Ker (C^\flat\overset{\nu_0}\epi C)\mbox{ for }n<0, \,\lim\limits_{n\to -\infty} x_n= 0.
\end{equation}

\begin{prop}   \label{p:M(C) to Acris(C)}
(i) The homomorphism $f:M(C)\to A_{\cris}(C)$ takes \eqref{e:2M(C) expansion} to 
\[
\sum_{m=0}^{\infty}V^m[x_m]+\sum_{l=1}^{\infty} (p^l-1)!\cdot\gamma_{p^l}([x_{-l}]).
\]

(ii) The composite map $M(C)\overset{f}\longrightarrow A_{\cris}(C)\overset{\beta}\epi W(C)$ is equal to the canonical epimorphism $M(C)\epi W(C)$.
\end{prop}

\begin{proof}
Statement (i) follows from Remark~\ref{r:(F')^n on Witt}. Statement (ii) follows from (i) or from Proposition~\ref{p:beta F'}.
\end{proof}

\subsubsection{The image of $f:M(C)\to A_{\cris}(C)$}  \label{sss:A''}
Let
\[
A'_{\cris}(C):=\{u\in A_{\cris}(C)\,|\,p^nu\in F^nA_{\cris}(C) \mbox{ for all }n\in\BN\}
\]
(Roughly, $A'_{\cris}(C)$ is the set of those $u\in A_{\cris}(C)$ for which $V^nu$ is defined for all $n$.)
Let $A''_{\cris}(C)$ be the set of all $u\in A_{\cris}(C)$ such that $p^nu=F^nv_n$ for some sequence $v_n\in A_{\cris}(C)$ converging to $0$. Clearly $A''_{\cris}(C)\subset A'_{\cris}(C)$.

\begin{lem}  \label{l:image of f}
$\im (M(C)\overset{f}\longrightarrow A_{\cris}(C))\subset A''_{\cris}(C)$.
\end{lem}

\begin{proof}
Let $y\in M(C)$. Then $p^nf(y)=f(F^nV^ny)=F^nf(V^ny)$. Moreover, $f(V^ny)\to 0$ because $V^ny\to 0$.
\end{proof}

\begin{prop}  \label{p:M(C)= A''_cris(C)}
Suppose that $C$ is as in \S\ref{sss:1thekeyexample}. Then $f:M(C)\to A''_{\cris}(C)$ is an isomorphism.
\end{prop}

\begin{proof}
We will work with the explicit description of $A_{\cris}(C)$ from \S\ref{sss:3thekeyexample}.
Since $F(x^\alpha )=x^{p\alpha}$, one gets the following description of the subsets 
$A''_{\cris}(C)\subset A'_{\cris}(C)\subset A_{\cris}(C)$. 

For $\alpha\in\BZ_+[1/p]^n$ let 
\begin{equation}   \label{e:m(alpha)}
m(\alpha ):=\max (0,\lfloor \log_p\alpha_1\rfloor ,\ldots , \lfloor \log_p\alpha_n\rfloor).
\end{equation}
Then $A''_{\cris} (C)$ (resp. $A'_{\cris} (C)$) identifies with the ring of formal series
\begin{equation}   \label{e:explicit A'cris}
\sum_{\alpha\in\BZ_+[1/p]^n} a_\alpha \frac{x^\alpha}{p^{m(\alpha )}} , \quad a_\alpha\in W(B) 
\end{equation}
such that $a_\alpha\to 0$ (resp.~$a_\alpha\to 0$ when $m(\alpha )$ is bounded).

Since $M(C)$ and $A''_{\cris} (C)$ are $\BZ_p$-flat topologically free $\BZ_p$-modules, it suffices to check that the map $\bar f:M(C)/pM(C)\to A''_{\cris}(C)/pA''_{\cris}(C)$ induced by $f$ is an isomorphism. By \eqref{e:explicit A'cris}, $A''_{\cris}(C)/pA''_{\cris}(C)$ is a free $B$-module with basis $y_\alpha$, $\alpha\in\BZ_+[1/p]^n$, where $y_\alpha\in A''_{\cris}(C)/pA''_{\cris}(C)$ is the image of $x^\alpha/p^{m(\alpha )}\in A''_{\cris}(C)$. On the other hand, by Proposition~\ref{p:M(C)}(vi), $M(C)/pM(C)$ identifies with $\gr C^\flat$, i.e., the associated graded of the decreasing filtration on $C^\flat$ whose $i$-th term equals $\Ker \nu_i$ if $n\ge 1$ and $C^\flat$ if $i\le 0$. It is clear that $\gr^i C^\flat$ is a free $B$-module with basis $x^\alpha$, where
$\alpha\in\BZ_+[1/p]^n$ is such that the number \eqref{e:m(alpha)} equals $i$. It is straightforward to check that $\bar f (x^\alpha )=y_\alpha$, so $\bar f$ is an isomorphism.
 \end{proof}

\section{The Dieudonn\'e functor according to Scholze-Weinstein and Fontaine}  \label{s:Dieudonne}
We fix a perfect field $k$ of characteristic $p$.

\subsection{Tate $k$-groups} 
\begin{defin}   \label{def:Tate}
A \emph{Tate $k$-group} is a group scheme $G$ over $k$ such that $\Ker (G\overset{p}\longrightarrow G)=0$, $\Coker (G\overset{p}\longrightarrow G)$ is finite, and the map 
$G\to \underset{n}{\underset{\longleftarrow}\lim} G/p^nG$ is an isomorphism. The category of all (resp. connected) Tate $k$-groups will be denoted by $\Tate_k$ (resp.~$\Tate_k^{\con}$).
\end{defin}

\begin{rem}  \label{r:Tate groups affine}
Any Tate $k$-group $G$ is affine because it is isomorphic to the projective limit of the finite group schemes $G/p^nG$.
\end{rem}

\subsubsection{Relation to $p$-divisible groups}  \label{sss:Tate&p-divisible}
If $H$ is a $p$-divisible group over $k$ then its Tate module
\[
\HHom (\BQ_p/\BZ_p ,H)=\HHom (\BZ [p^{-1}]/\BZ ,H)=\underset{n}{\underset{\longleftarrow}\lim}\, \HHom (p^{-n}\BZ/\BZ ,H)
\]
is a Tate $k$-group. Thus we get an equivalence between the category of $p$-divisible groups and $\Tate_k$. The inverse equivalence takes $G\in\Tate_k$ to
\[
\underset{\longrightarrow}\lim (G/pG\overset{p}\longrightarrow G/p^2G\overset{p}\longrightarrow G/p^3G\overset{p}\longrightarrow\ldots ).
\]
It identifies the full subcategory $\Tate_k^{\con}\subset\Tate_k$  with the category of connected $p$-divisible groups.

\begin{rem}  \label{r:Tate module via Fr}
If $H$ is a \emph{connected} $p$-divisible $k$-group then its Tate module $\HHom (\BZ [p^{-1}]/\BZ ,H)$ can also be described as
$\Ker (H_{\perf}\to H)$, where 
\[
H_{\perf}=\underset{\longleftarrow}\lim (\ldots \overset{\Fr}\longrightarrow H\overset{\Fr}\longrightarrow H\overset{\Fr}\longrightarrow H).
\]
Indeed, in the connected case $H_{\perf}=\underset{n}{\underset{\longleftarrow}\lim}\, \HHom (p^{-n}\BZ ,H)=\HHom (\BZ [p^{-1}], H)$, so
\[
\Ker (H_{\perf}\to H)=\Ker (\HHom (\BZ [p^{-1}], H)\to\HHom (\BZ, H))=\HHom (\BZ [p^{-1}/\BZ, H).
\]

\end{rem}

\begin{rem}  \label{r:connected times etale}
By \S\ref{sss:Tate&p-divisible}, any Tate $k$-group can be uniquely represented as a direct product of a connected Tate group and a reduced one. Moreover, any reduced Tate $k$-group is a projective limit of etale group schemes.  
\end{rem}

\subsection{Tate $k$-groups as schemes}  \label{ss:Tate as schemes}
Any Tate $k$-group $G$ is a semiperfect scheme. Indeed, 
$$\Ker (G\overset{\Fr}\longrightarrow G)\subset\Ker (G\overset{p}\longrightarrow G)=0,$$
so $\Fr :G\to G$ is a closed embedding.

\begin{prop}  \label{p:Tate as schemes}
The underlying scheme of any connected Tate $k$-group is isomorphic to 
$$\Spec k[x_1^{p^{-\infty}},\ldots ,x_n^{p^{-\infty}}]/(x_1,\ldots , x_n)$$ 
for some $n$.
\end{prop}

\begin{proof}
Let $H$ be a connected $p$-divisible $k$-group. As an ind-scheme, $H$ is isomorphic to 
$\Spf k[[x_1,\ldots , x_n]]$ for some $n$. So the proposition follows by Remark~\ref{r:Tate module via Fr}.
\end{proof}

\subsection{The Dieudonn\'e functor according to Scholze-Weinstein}   \label{ss:according to SW}
Let $D_k$ be the Dieudonn\'e ring of $k$ (i.e., the ring generated by $W(k)$ and elements $F,V$ subject to the usual relations).  Let $W(k)[F]\subset D_k$ be the subring generated by $W(k)$ and $F$. 

Let $\fD_k$ be the category of those $D_k$-modules that are free and finitely generated over~$W(k)$. Classical Dieudonn\'e theory \cite[Ch.~III]{Dem} provides an equivalence 
 \[
 \fD_k^{\op}\iso\Tate_k , \quad N\mapsto G_N\, .
 \]
 
 \begin{defin}   \label{d:Good}
$\Good_k$ is the category of $k$-algebras isomorphic to $$\Spec B[x_1^{p^{-\infty}},\ldots ,x_n^{p^{-\infty}}]/(x_1,\ldots , x_n)$$ 
for some perfect $k$-algebra $B$.
 \end{defin}

The following theorem is due to P.~Scholze and J.~Weinstein. 

\begin{thm}  \label{t:SW}
There is a canonical isomorphism of functors
\begin{equation}  \label{e:SW}
G_N(C)\iso \Hom_{W(k)[F]}(N,A_{\cris} (C)), \quad C\in \Good_k\, , N\in\fD_k \,.
\end{equation}
\end{thm}

\begin{rem}   \label{r:Good suffices}
By Proposition~\ref{p:Tate as schemes} and Remark~\ref{r:connected times etale}, the underlying scheme of any connected Tate $k$-group is isomorphic to $\Spec C$, where $C\in\Good_k\,$. So $G_N$ is completely determined by the functor $C\mapsto G_N(C)$, $C\in\Good_k\,$.
\end{rem}

\begin{rem}  \label{r:Acris overkill}
In  \S\ref{sss:A''} we defined $W(C^\flat )[F]$-submodules $A''_{\cris} (C)\subset A'_{\cris} (C)\subset A_{\cris} (C))$.
Using the map $V:N\to N$, one sees that 
$$\Hom_{W(k)[F]}(N,A_{\cris} (C))=\Hom_{W(k)[F]}(N,A'_{\cris} (C));$$ 
moreover, if $V:N\to N$ is topologically nilpotent then
$$\Hom_{W(k)[F]}(N,A_{\cris} (C))=\Hom_{W(k)[F]}(N,A''_{\cris} (C)).$$ 
Note that $A'_{\cris} (C)$ is much smaller than $A_{\cris} (C)$: compare the denominators in \eqref{e:explicit Acris} and \eqref{e:explicit A'cris}.
\end{rem}

Theorem \ref{t:SW} is a very special case of the theory\footnote{The theory from \cite[\S 4]{SW} was refined in \cite{Lau}.} developed in 
\cite[\S 4]{SW}. More precisely, it is deduced from \cite[Cor.~4.1.12]{SW} as follows. Let $H$ be the $p$-divisble group corresponding to $G_N$; then
$G_N(C)=\Hom (\BQ_p/\BZ_p)_C,H_C)$, where $(\BQ_p/\BZ_p)_C$ and $H_C$ are the constant $p$-divisble groups over $\Spec C$ with fibers $\BQ_p/\BZ_p$ and $H$. To get \eqref{e:SW}, it suffices to compute $\Hom (\BQ_p/\BZ_p)_C,H_C)$ using \cite[Cor.~4.1.12]{SW}.

In what follows we give a \emph{self-contained} construction of \eqref{e:SW}.

Recall that $\fD_k=\fD'_k\oplus \fD''_k$, where $\fD'_k $ (resp.~$\fD''_k$) is the full subcategory of objects $N\in\fD_k$ such that the operator $V:N\to N$ is topologically nilpotent (resp.~invertible). In \S\ref{ss:N pronilpotent} and \S\ref{ss:multiplicative Tate} we will construct the isomorphism \eqref{e:SW} for $N\in\fD'_k$ and $N\in\fD''_k$, respectively. Let us note that in \S\ref{ss:multiplicative Tate} we just paraphrase the relevant part of \cite[\S 4]{SW}. In \S\ref{ss:conceptually better} we sketch a conceptually better proof of Theorem~\ref{t:SW}, which treats $\fD'_k$ and $\fD''_k$ simultaneously.

\begin{rem}
The category $\Good_k$ from Definition~\ref{d:Good} is contained in the following category $\Good'_k$ introduced in  \cite[\S 4]{SW}:  a $k$-algebra $C$ is in $\Good'_k$ if 
$C$ can be represented as a quotient of a perfect ring by an ideal generated by a regular sequence.
In \cite[\S 4]{SW} the isomorphism \eqref{e:SW} is established for $C\in\Good'_k$. I do not know if the construction of \eqref{e:SW} given below works for $C\in\Good'_k$
(I did not check whether Propositions~\ref{p:M(C)= A''_cris(C)} and \ref{p:bar f is an iso} hold in the more general setting).
\end{rem}

\subsection{The isomorphism~\eqref{e:SW} for $N\in\fD'_k$}  \label{ss:N pronilpotent}
If $N\in\fD'_k$ then $G_N:=\HHom_{D_k}(N,W)$, where $W$ is the Witt group scheme. This means that 
\[
G_N(C)=\Hom_{D_k}(N,W(C))
\]
for $N\in\fD'_k$ and any $k$-algebra $C$.

\begin{prop}   \label{p:N pronilpotent}
Let $N\in\fD'_k$ and $C\in \Good_k\,$. Then the canonical epimorphism $A_{\cris} (C)\to W(C)$ (see \S\ref{sss:beta}) induces an isomorphism 
\[
\Hom_{W(k)[F]}(N,A_{\cris}(C))\iso\Hom_{D_k}(N,W(C))\subset\Hom_{W(k)[F]}(N,W(C)).
\]
\end{prop}

\begin{proof}
In \S\ref{sss:A''} we defined a $W(C^\flat )[F]$-submodule $A''_{\cris}(C))\subset A_{\cris}(C)$.
We have $$\Hom_{W(k)[F]}(N,A_{\cris} (C))=\Hom_{W(k)[F]}(N,A''_{\cris} (C))$$  by Remark~\ref{r:Acris overkill} and the assumption $N\in\fD'_k\,$.

In \S\ref{sss:M(C)} and  \S\ref{sss:M(C) to Acris(C)} we defined a $W(C^\flat )[F,V]$-module $M(C)$ and an
$W(C^\flat )[F]$-morphism $f:M(C)\to A_{\cris}(C)$. By Proposition~\ref{p:M(C)= A''_cris(C)} and the assumption $C\in\Good_k$, this morphism induces an isomorphism $M(C)\iso A''_{\cris}(C)$. So
$$\Hom_{W(k)[F]}(N,A_{\cris}(C))=\Hom_{W(k)[F]}(N,M(C)).$$
By Proposition~\ref{p:M(C)}(iii), $M(C)$ is $\BZ_p$-flat, so
$$\Hom_{W(k)[F]}(N,A_{\cris}(C))=\Hom_{D_k}(N,M(C)).$$

It remains to show that the map 
\begin{equation} \label{e:bijection to check}
\Hom_{D_k}(N,M(C))\to\Hom_{D_k}(N,W(C))
\end{equation}
induced by the canonical epimorphism $M(C)\epi W(C)$ is bijective.
In fact, we will prove this without assuming that $C\in \Good_k$. 

Let us describe the preimage of $\alpha\in\Hom_{D_k}(N,W(C))$ under the map \eqref{e:bijection to check}.
First note that for each $m,n\ge 0$ we have a map
\[
V^{-m}F^n:W(C)\to W(C^\flat /\Ker\nu_n)[V^{-1}],
\]
where $\nu_n: C^\flat\epi C$ is as in \S\ref{ss:Fontaine's def}. If $m$ is so big that $V^mN\subset F^nN$ then we get a map
\begin{equation}  \label{e: the map modulo Ker nu_n}
N\to W(C^\flat /\Ker\nu_n)[V^{-1}], \quad x\mapsto V^{-m}F^n \alpha (F^{-n}V^m x).
\end{equation}
It is easy to check that this map does not depend on $m$; moreover, the maps \eqref{e: the map modulo Ker nu_n} corresponding to different $n$ agree with each other and therefore define a map 
$$N\to M(C)\subset\underset{n}{\underset{\longleftarrow}\lim} W(C^\flat /\Ker\nu_n)[V^{-1}].$$ 
It is easy to check that this is the unique preimage of $\alpha$ under the map  \eqref{e:bijection to check}.
\end{proof}

\subsection{The isomorphism~\eqref{e:SW} for $N\in\fD''_k$}      \label{ss:multiplicative Tate}
In this subsection we follow \cite[\S 4]{SW}.

Let $H\in\Tate_k$ be the projective limit of the group schemes $\mu_{p^m}$, $m\in\BN$; then
\[
H(C)=\Ker((C^\flat )^\times\to C^\times )
\] 
for any $k$-algebra $C$. 

Fix an algebraic closure $\bar k$. For  $N\in\fD''_k$ one has
\begin{equation}   \label{e:G_N for mult}
G_N(C):=\Hom_{\Gal (\bar k/k)}(N_0, H(C\otimes_k\bar k)),
\end{equation}
where $N_0$ is the following $\BZ_p$-module with an action of $\Gal (\bar k/k)$:
\[
N_0:=(N\otimes_{W(k)}W(\bar k))^{V=1}:=\Ker (N\otimes_{W(k)}W(\bar k)\overset{V-1}\longrightarrow N\otimes_{W(k)}W(\bar k)).
\]

\begin{prop}  \label{p:SW on log}
Let $C\in \Good_k$. Then the map
\[
H(C)=\Ker((C^\flat )^\times\to C^\times )\to A_{\cris}(C), \quad c\mapsto \log [c]:=-\sum_{n=1}^\infty (n-1)!\cdot \gamma_n(1-[c])
\]
induces an isomorphism 
\begin{equation}    \label{e:SW on log}
H(C)\iso A_{\cris}(C)^{F=p}:=\Ker (A_{\cris}(C)\overset{F-p}\longrightarrow A_{\cris}(C)).
\end{equation}
\end{prop}

In \S\ref{sss:SW on log} we will give a proof of the proposition following  \cite[\S 4]{SW} (where a more general statement is proved). Using Proposition~\ref{p:SW on log}, one gets the isomorphism~\eqref{e:SW} as follows.
Combining \eqref{e:G_N for mult} and \eqref{e:SW on log}, one gets a canonical isomorphism
\begin{equation}    \label{e:2SW on log}
G_N(C)=\Hom_{\Gal (\bar k/k)}(N_0, A_{\cris}(C\otimes_k\bar k)^{F=p}).
\end{equation}
We also have the canonical isomorphisms 
\begin{equation}    \label{e:3SW on log}
N_0\otimes_{\BZ_p}W(\bar k)\iso N\otimes_{W(k)}W(\bar k), \quad A_{\cris}(C)\hat\otimes_{W(k)}W(\bar k )\iso A_{\cris}(C\otimes_k\bar k)
\end{equation}
(the second one by Proposition~\ref{p:flat perfect base change}). Combining \eqref{e:2SW on log}, and \eqref{e:3SW on log},
we see that an element of $G_N(C)$ is the same as a $\Gal (\bar k/k)$-equivariant $D_{\bar k}$-homomorphism
$N\otimes_{W(k)}W(\bar k)\to A_{\cris}(C)\hat\otimes_{W(k)}W(\bar k )$, which is the same as a $D_k$-homomorphism $N\to A_{\cris}(C)$.

\subsubsection{Proof of Proposition~\ref{p:SW on log}}  \label{sss:SW on log}
If $c\in\Ker((C^\flat )^\times\to C^\times )$ then $$F(\log [c])=\log F([c])=\log ([c]^p)=p\log [c].$$

Suppose that $c\in\Ker((C^\flat )^\times\to C^\times )$ and $\log [c]=0$. Note that for $u\in\Ker (W(C^\flat )\to C)$ one has $(1+u)^p-1\in pA_{\cris}(C)$. In particular, $[c]^p\in 1+pA_{\cris}(C)$, so $[c]^{p^2}\in 1+p^2A_{\cris}(C)$. Since 
$\log ([c]^{p^2})=0$ and $[c]^{p^2}\in 1+p^2A_{\cris}(C)$, a standard argument\footnote{One uses that the formal series $x^{-1}\cdot \log (1+p^2x)$ belongs to $(\BZ_p[[x]])^\times$.} shows that $[c]^{p^2}$ equals $1$ in $A_{\cris}(C)$. But our assumption on $C$ implies that the map $W(C^\flat )\to A_{\cris}(C)$ is injective, so $[c]^{p^2}$  equals $1$ in  $W(C^\flat )$.
Therefore $c^{p^2}=1$. By perfectness of $C$, this implies that $c=1$. This proves injectivity of the map \eqref{e:SW on log}. 

The proof of surjectivity from \cite[\S 4.2]{SW} uses the Artin-Hasse exponential
\[
E_p(y):=\exp (y+\frac{y^p}{p}+\frac{y^{p^2}}{p^2}+\ldots )\in\BZ_p[[y]].
\]
In \cite[\S 4.2]{SW} it is proved that
\begin{equation}  \label{e:E_p}
\log [E_p(c)]=\sum_{n=-\infty}^{\infty}\frac{[c^{p^n}]}{p^n}:=\sum_{m=0}^\infty p^m\cdot [c^{p^{-m}}]+\sum_{n=1}^{\infty} (p^n-1)!\cdot\gamma_{p^n}([c]), \quad c\in\Ker (C^\flat\epi C).
\end{equation}
(a way to think about the r.h.s is explained in \S\ref{sss:remark on E_p} below). 
Formula~\eqref{e:E_p} is an immediate consequence of the formula
\begin{equation}  \label{e:2E_p}
[E_p(c)]=E_p([c])\cdot\exp\sum_{m=1}^\infty p^m [c]^{p^{-m}}, \quad c\in\Ker (C^\flat\epi C).
\end{equation}
which is proved in \cite[\S 4.2]{SW} as follows. First, note that the exponent in the r.h.s. of \eqref{e:2E_p} has the form $\exp (py)$, where $y\in W(C^\flat )$ is topologically nilpotent; so the exponent makes sense as an element of $W(C^\flat )$. Let $z\in W(C^\flat )$ be the r.h.s. of \eqref{e:2E_p}, then it is easy to check that $F(z)=z^p$ and the image of $z$ in $C^\flat$ equals $E_p(c)$. These properties of $z$ mean that $z=[E_p(c)]$.

Let us now prove surjectivity. Let $u\in A_{\cris}(C)$ be such that $Fu=pu$; we want to represent $u$ as $\log [c]$, $c\in\Ker ((C^\flat)^\times\to C^\times)$.
Fix an isomorphism $$C\iso B[x_1^{p^{-\infty}},\ldots ,x_n^{p^{-\infty}}]/(x_1,\ldots , x_n),$$ where $B$ is a perfect $k$-algebra. Using the realization of $A_{\cris}(C)$ from \S\ref{sss:3thekeyexample}, write
\begin{equation} \label{e:2explicit Acris}
u=\sum_{\alpha\in\BZ_+[1/p]^n} a_\alpha \frac{x^\alpha}{(\alpha !)_p} , \quad a_\alpha\in W(B), \quad a_\alpha\to 0,
\end{equation}
where $x^\alpha$ really means $[x_1^{\alpha_1}]\cdot\ldots\cdot [x_n^{\alpha_n}]$. 
Since $Fu=pu$, it is clear that $u$
is an infinite $W(B)$-linear combination of elements as in the r.h.s. of \eqref{e:E_p}. So using \eqref{e:E_p}, it is straightforward to find $c\in\Ker ((C^\flat)^\times\to C^\times)$ such that $u=\log [c]$. To formulate the answer, we will use the standard monomorphism
\[
W(B)\mono B[[t]]^\times , \quad a\mapsto f^a;
\]
recall that if $a=\sum\limits_{i=0}^\infty V^i[b_i]$, $b_i\in B$, then $f^a(t):=\prod\limits_{i=0}^\infty E_p(b_it^{p^i})$.
Here is the formula for $c$, which follows from  \eqref{e:E_p}:
\begin{equation} \label{e:explicit c}
c=\prod_{\alpha\in S} f^{a_\alpha}(x^\alpha )\in \Ker((C^\flat )^\times\to C^\times ), \quad\mbox{ where } S:=[0,p)^n\setminus [0,1)^n\subset\BZ_+[1/p]^n .
\end{equation}
The product converges because $a_\alpha\to 0$.

\subsection{Another approach to the proof of Theorem~\ref{t:SW}}   \label{ss:conceptually better}
In \S\ref{ss:N pronilpotent}-\S\ref{ss:multiplicative Tate} we treated $\fD'_k $ and $\fD''_k$ separately, which is not good philosophically;
this is related to the fact that we used the definition of the functor $N\mapsto G_N$ from \cite{Dem}, which has a similar drawback. Here is a sketch of a conceptually better proof of Theorem~\ref{t:SW}, which uses the description of $G_N$ given by  Fontaine \cite{F77}. This description is recalled in \S\ref{sss:Fontaine's description} below.

\subsubsection{Witt covectors and bivectors}
For any discrete $\BF_p$-algebra $R$ we introduced in \S\ref{sss:M(C)} the $W(R^\flat )$-module $W(R)[V^{-1}]$. Each element of $W(R)[V^{-1}]$ has a unique expansion 
\begin{equation}   \label{e:Witt bivector}
\sum_{m=-\infty}^{\infty}V^m[x_m], \quad x_m\in R,
\end{equation}
where $x_m=0$ for sufficiently negative $m$.

In \cite[\S V.1.3]{F77} Fontaine defines the \emph{group of Witt bivectors} $BW (R)$, which  contains $W(R)[V^{-1}]$ as a subgroup. Elements of $BW (R)$ are formal expressions \eqref{e:Witt bivector} such that for some $N< 0$ the ideal generated by $x_N, x_{N-1}, x_{N-2}, \ldots$ is \emph{nilpotent}. The group $BW (R)/W(R)$ is denoted by $CW (R)$ and called the \emph{group of Witt covectors.}

In fact, the definition of  $BW (R)$ from \cite[p.~228]{F77} \emph{relies} on that of $CW (R)$. The latter is given in \cite[Ch.~II.1]{F77} and is based on  Proposition 1.1 and Lemma 1.2 of \cite[Ch.~II]{F77}.

If $R$ is an algebra over a perfect $\BF_p$-algebra $k$ then $BW (R)$ and $CW(R)$ are $W(k)$-modules by an argument similar to Proposition II.2.2 of \cite{F77} (in which there is an extra assumption that $k$ is a field). In particular, $BW (R)$ and $CW(R)$ are $W(R^\flat )$-modules.

\subsubsection{The module $\overline{M} (C)$}   \label{sss:bar M(C)}
Let $C$ be a semiperfect $\BF_p$-algebra. Recall that $C^\flat$ is the projective limit of $C^\flat/\Ker\nu_n $, where $\nu_n$ is as in \S\ref{ss:Fontaine's def}.
We equip $C^\flat$ with the projective limit topology. Define $BW (C^\flat )$ to be the projective limit of $BW (C^\flat/\Ker\nu_n )$. 
Following p.~229 of Fontaine's book \cite{F77}, consider the preimage of 
$W(C)\subset BW(C)=BW(C^\flat/\Ker\nu_0)$ in $BW (C^\flat )$; we denote\footnote{Fontaine's notation for $\overline{M} (C))$ is $BW_0(\kappa (C))$. Here $\kappa (C)$ is our $C^\flat$.} this preimage by $\overline{M} (C)$. Thus elements of $\overline{M} (C)$ are formal expressions
\begin{equation}   \label{e:2Witt bivector}
\sum_{m=-\infty}^{\infty}V^m[x_m], \quad x_m\in C^\flat,\quad x_m\in\Ker(C^\flat \overset{\nu_0}\longrightarrow C) \mbox{ for } m<0
\end{equation}
such that the ideal in $C^\flat$ generated by $x_{-1}$, $x_{-2}$, \ldots is topologically nilpotent\footnote{Here we are using that $x_{-i}$ is topologically nilpotent for each $i>0$.}. The topological nilpotence condition is automatic if $C$ is as in the following lemma.

\begin{lem}   \label{l:nilpotence equivalence}
The following properties of $C$ are equivalent:

(i) the ideal $\Ker (C\overset{\Fr}\longrightarrow C)$ is nilpotent;

(ii) the ideal $\Ker (C\overset{\Fr^n}\longrightarrow C)$ is nilpotent.
\end{lem}

\begin{proof}
Suppose that the product of any $m$ elements of $\Ker (C\overset{\Fr}\longrightarrow C)$  is zero. Then the products of any $m$ elements of $\Ker (C\overset{\Fr^n}\longrightarrow C)$  belongs to $\Ker (C\overset{\Fr^{n-1}}\longrightarrow C)$ . So the product of any $m^n$ elements of $\Ker (C\overset{\Fr^n}\longrightarrow C)$  is zero.
\end{proof}

Clearly $\overline{M}(C)$ is a $W(C^\flat )$-module equipped with maps $F,V:\overline{M}(C)\to\overline{M}(C)$ satisfying the usual identities.
Note that $\overline{M} (C)\supset M(C)$, where $M(C)$ is as in \S\ref{sss:M(C)}.

Similarly to Proposition~\ref{p:M(C)}, one proves the following

\begin{lem}   \label{l:overline M(C)}
(i) $\overline{M} (C)$ is $p$-adically complete.

(ii) $\Ker (\overline{M}(C)\overset{p}\longrightarrow \overline{M}(C))=0$.

(iii) If $C$ has the properties from Lemma~\ref{l:nilpotence equivalence} then one has a canonical $W(C^\flat )$-module isomorphism 
$$\overline{M}(C)/p\overline{M}(C)\iso\prod\limits_{i=0}^\infty\gr^i C^\flat,$$ 
where $\gr C^\flat$ corresponds to the filtration $C^\flat\supset\Ker\nu_1\supset\Ker\nu_2\supset\ldots$\;. \qed
\end{lem}

\subsubsection{Fontaine's description of $G_N$}   \label{sss:Fontaine's description}
Let $k$ be a perfect field and $N\in\fD_k$. Proposition~V.1.2 of \cite{F77} tells us 
that for any semiperfect $k$-algebra $C$ one has
\begin{equation}  \label{e:Fontaine's description}
G_N(C)=\Hom_{D_k}(N,\overline{M} (C)).
\end{equation}
Note that by Lemma~\ref{l:overline M(C)}(ii), one has
\begin{equation}    \label{e:2Fontaine's description}
\Hom_{D_k}(N,\overline{M} (C))=\Hom_{W(k)[F]}(N,\overline{M} (C)).
\end{equation}

\subsubsection{The homomorphism $\bar f:\overline{M}(C)\to A'_{\cris}(C)$}   \label{sss:bar f}
The homomorphism $f:M(C)\to A_{\cris}(C)$ from \S\ref{sss:M(C) to Acris(C)} and Proposition~\ref{p:M(C) to Acris(C)} canonically extends to a homomorphism of $W(C^\flat )[F]$-modules $\bar f:\overline{M}(C)\to A_{\cris}(C)$; namely, $\bar f$ takes an element \eqref{e:2Witt bivector} to 
$$\sum\limits_{m=0}^{\infty}V^m[x_m]+\sum\limits_{l=1}^{\infty} (p^l-1)!\cdot\gamma_{p^l}([x_{-l}])\in A_{\cris}(C).$$
Similarly to Lemma~\ref{l:image of f} one shows that $\bar f (\overline{M}(C))\subset A'(C)$, where $A'(C)$ is defined in \S\ref{sss:A''}.

Similarly to Proposition~\ref{p:M(C)= A''_cris(C)}, one deduces from Lemma~\ref{l:overline M(C)} the following

\begin{prop}   \label{p:bar f is an iso}
If $C$ is as in \S\ref{sss:1thekeyexample} then the map $\bar f:\overline{M} (C)\to A'_{\cris}(C)$ is an isomorphism. \qed
\end{prop}

\subsubsection{Proof of Theorem~\ref{t:SW}}
By Remark~\ref{r:Acris overkill}, we have $$\Hom_{W(k)[F]}(N,A_{\cris} (C))=\Hom_{W(k)[F]}(N,A'_{\cris} (C)).$$
Combining this with \eqref{e:Fontaine's description}-\eqref{e:2Fontaine's description} and Proposition~\ref{p:bar f is an iso}, we get the desired isomorphism
\[
G_N(C)\iso \Hom_{W(k)[F]}(N,A_{\cris} (C)), \quad C\in \Good_k\, , N\in\fD_k \,.
\]

\subsubsection{A remark on the proof of Proposition~\ref{p:SW on log}}   \label{sss:remark on E_p}
The r.h.s. of formula \eqref{e:E_p} equals $\bar f(w)$, where $w\in\overline{M}(C)^{V=1}$ is defined by $w:=\sum\limits_{n=-\infty}^\infty V^n[c]$ and $\bar f$ is as in \S\ref{sss:bar f}.

\bibliographystyle{alpha}

\begin{thebibliography}{BFM}
\bibitem[B74]{B74}
P.~Berthelot, Cohomologie cristalline des sch\'emas de caract\'eristique $p>0$. 
Lecture Notes in Mathematics, {\bf 407}, Springer-Verlag, Berlin-New York, 1974.



\bibitem[BMS]{BMS} B.~Bhatt, M.~Morrow, and P.~Scholze,
\textit{Topological Hochschild homology and integral p-adic Hodge theory}, arXiv:1802.03261.

\bibitem[Dem]{Dem} M.~Demazure, \textit{Lectures on $p$-divisible groups}, Lecture Notes in Math. {\bf 302}, Springer-Verlag, Berlin-New York, 1972.


\bibitem[F77]{F77} J.-M.~Fontaine, \textit{Groupes p-divisibles sur les corps locaux}, Ast\'erisque {\bf 47-48}, Soc. Math. France, Paris, 1977.

\bibitem[F94]{F94} J.-M.~Fontaine, \textit{Le corps des p\'eriodes p-adiques}. In: P\'eriodes p-adiques, 59--111, Ast\'erisque {\bf 223},  Soc. Math. France, Paris, 1994.

\bibitem[FF]{FF} L.~Fargues and J.-M. Fontaine, \textit{Courbes et fibr\'es vectoriels en
  th\'eorie de {H}odge $p$-adique}, 2017, \\
  https://webusers.imj-prg.fr/$\sim$laurent.fargues/Courbe\_fichier\_principal.pdf.



\bibitem[Lau]{Lau} E.~Lau, \textit{Dieudonn\'e theory over semiperfect rings and perfectoid rings}, arXiv:1603.07831.


\bibitem[SW]{SW} P.~Scholze and J.~Weinstein, \textit{Moduli of p-divisible groups}, Camb. J. Math. {\bf 1} (2013), no. 2, 145--237. 

\end{thebibliography}

\end{document}